\documentclass{amsart}

\usepackage{amssymb}
\usepackage{graphicx,color}
\usepackage{amsmath}
\usepackage{url}
\usepackage[colorlinks=true]{hyperref}

\usepackage{cleveref}
\Crefname{counterexample}{counterexample}{counterexamples}
\Crefname{counterexample}{Counterexample}{Counterexamples}
\Crefname{conjecture}{conjecture}{conjectures}
\Crefname{conjecture}{Conjecture}{Conjectures}

\usepackage{tikz-cd}
\usepackage{tikz}
\usetikzlibrary{matrix}
\usepackage{tkz-euclide}
\usetikzlibrary{intersections}
\usetkzobj{all}

\theoremstyle{plain}
\newtheorem{theorem}{Theorem}[section]
\newtheorem{proposition}[theorem]{Proposition}
\newtheorem{lemma}[theorem]{Lemma}
\newtheorem{corollary}[theorem]{Corollary}
\newtheorem{claim}[theorem]{Claim}

\theoremstyle{definition}
\newtheorem{definition}[theorem]{Definition}

\newtheorem{example}[theorem]{Example}

\newtheorem{question}[theorem]{Question}
\newtheorem{conjecture}[theorem]{Conjecture}

\newcommand{\Z}{ \ensuremath{\mathbb{Z}}}

\newcommand{\R}{ \ensuremath{\mathbb{R}}}

\renewcommand{\int}{\operatorname{int}}
\newcommand{\width}{\operatorname{width}}

\newcommand{\area}{\operatorname{area}}

\renewcommand{\vec}[1]{\overrightarrow#1}
\newcommand{\vecline}[1]{\langle \vec #1 \rangle}

\newcommand{\conv}{\ensuremath{\mathrm{conv}}\hspace{1pt}}
\newcommand{\cayley}{\operatorname{Cay}}

\usepackage[colorinlistoftodos]{todonotes}

\date{\today}

\author[G.~Codenotti]{Giulia Codenotti}
\author[F.~Santos]{Francisco Santos}

\address[G.~Codenotti]
{
Institut f\"ur Mathematik, Freie Universit\"at Berlin, Germany
}
\email{giulia.codenotti@fu-berlin.de}

\address[F.~Santos]
{
Department of Mathematics, Statistics and Computer Science, University of Cantabria, Spain
}
\email{francisco.santos@unican.es}

\thanks{The authors were supported by the Einstein Foundation Berlin under grant EVF-2015-230. 
Work of F. Santos is also supported by project MTM2017-83750-P of the Spanish Ministry of Science (AEI/FEDER, UE)}

\title{Unimodular covers of 3-dimensional parallelepipeds and Cayley sums}

\begin{document}

\begin{abstract}
We show that the following classes of lattice polytopes have unimodular covers, in dimension three: the class of parallelepipeds, the class of centrally symmetric polytopes, and the class of Cayley sums $\cayley(P,Q)$ where the normal fan of $Q$ refines that of $P$. This improves results of Beck et al.~(2018) and Haase et al.~(2008) where the last two classes were shown to be IDP.
\end{abstract}

\maketitle


\section{Introduction}

A lattice polytope $P\subset \R^d$ has the \emph{integer decomposition property} if for every positive integer $n$, every lattice point $p \in nP\cap \Z^d$ can be written as a sum of $n$ lattice points in $P$. We abbreviate this by saying that ``$P$ is IDP''. Being IDP is interesting in the context of  both enumerative combinatorics (Ehrhart theory) and algebraic geometry (normality of toric varieties). It falls into a hierarchy of several properties each stronger than the previous one; see, e.g., \cite[Section 2.D]{BGbook}, \cite[Sect. 1.2.5]{HPPS-survey}, \cite[p. 2097]{mfo2004}, \cite[p. 2313]{mfo2007}.
Let us here only mention that
\[
P \text{ has a unimodular triangulation}\Rightarrow
P \text{ has a unimodular cover}\Rightarrow
P \text{ is IDP.}
\]
Remember that a \emph{unimodular triangulation} is a triangulation of $P$ into unimodular simplices, and a \emph{unimodular cover} is a collection of unimodular simplices whose union equals $P$. 

Oda (\cite{Oda1997}) posed several questions regarding smoothness and the IDP property for lattice polytopes.
Following \cite{HaaseHof, Tsuchiya}, we say that a pair $(P, Q)$ of lattice polytopes has the integer decomposition property, or that \emph{the pair $(P,Q)$ is IDP}, if 
\begin{align*}
\label{eq:mixedIDP}
(P+Q) \cap \Z^d = P \cap \Z^d + Q \cap \Z^d.
\end{align*}
A lattice polytope $Q$ is called \emph{smooth} if it is simple and the primitive edge directions at every vertex form a linear basis for the lattice; equivalently, if the projective toric variety defined by the normal fan of $Q$ is smooth. 
The following versions of Oda's questions are now considered conjectures~\cite{HNPS2008,mfo2007}, and they are open even in dimension three:
\begin{conjecture}
\label{conj:Oda}
\begin{enumerate}
\item 
\label{itm:smoothIDP}
(Related to problems 2 and 5 in \cite{Oda1997})
Every smooth lattice polytope is IDP.
\item 
\label{itm:mixedIDP}
(Related to problems 1, 3, 4, 6 in \cite{Oda1997}) Every pair $(P,Q)$ of lattice polytopes with $Q$ smooth and the normal fan of $Q$ refining that of $P$ is IDP.
\end{enumerate}
\end{conjecture}

When the normal fan of a polytope $Q$ refines that of another polytope $P$, as in the second conjecture, we say that $P$ \emph{is a weak Minkowski summand of $Q$}, since this is easily seen to be equivalent to the existence of a polytope $P'$ such that $P+P' = k Q$ for some dilation constant $k>0$.
This property has the following algebraic implication for the projective toric variety $X_Q$: $P$ is a weak Minkowski summand of $Q$ if and only if the Cartier divisor defined by $P$ on  $X_Q$ is \emph{numerically effective}, or ``nef'' (see~\cite[Cor.~6.2.15, Prop.~6.3.12]{CLS}, but observe that what we here call  ``weak Minkowski summand'' is simply called ``Minkowski summand'' there).

\medskip 
Motivated by these and other questions, several authors have studied the IDP property for different classes of lattice polytopes. 
For example,  very recently
Beck et al.~\cite{BHHHJKM2019} proved that all smooth centrally symmetric $3$-polytopes are IDP.
More precisely, they show that any such polytope can be covered by lattice 
parallelepipeds and unimodular simplices, both of which are trivially IDP.
In \Cref{sec:parallelepipeds} we show:

\begin{theorem}
\label{thm:parallelepipeds}
Every $3$-dimensional lattice parallelepiped has a unimodular cover.
\end{theorem}

This, together with the mentioned result from~\cite{BHHHJKM2019}, gives:

\begin{corollary}
\label{coro:3cs}
Every smooth centrally symmetric lattice $3$-polytope has a unimodular cover. 
\qed
\end{corollary}

These results leave open the following important questions:

\begin{question}
Do $3$-dimensional parallelepipeds have unimodular triangulations?
\end{question}

\begin{question}
Higher dimensional parallelotopes (affine images of cubes) are IDP. Do they have unimodular covers? 
\end{question}

The two-dimensional case of \Cref{conj:Oda}\eqref{itm:mixedIDP} is known to hold, with three different proofs by Fakhruddin~\cite{Fakhruddin}, Ogata~\cite{Ogata} and Haase et al.~\cite{HNPS2008}. This last one actually shows that smoothness of $Q$ is not needed. In dimension three, however, the conjecture fails without the smoothness assumption. Indeed, if we let $P=Q$ be any non-unimodular \emph{empty tetrahedron}, then $P$ is obviously a weak Minkowski summand of $Q$ but the pair $(P.Q)$ is not IDP. By an empty tetrahedron we mean a lattice tetrahedron containing no lattice points other than its vertices (see the proof of \Cref{lemma:corner} for a classification of them).

An alternative approach to \Cref{conj:Oda}\eqref{itm:mixedIDP} is via Cayley sums, which we discuss in  \Cref{sec:cayley}. 
Recall that the \emph{Cayley sum} of two lattice polytopes $P,Q\subset \R^d$ is the lattice polytope
\[
\cayley(P,Q) := \conv(P\times\{0\} \cup Q \times \{1\}) \subset \R^3.
\]
We normally require $\cayley(P,Q)$ to be full-dimensional (otherwise we can delete coordinates) but this does not need $P$ and $Q$ to be full-dimensional. It only requires the linear subspaces parallel to them to span $\R^d$.

As we note in \Cref{prop:mixedIDP}, if the Cayley sum of $P$ and $Q$ is IDP then the pair $(P,Q)$ is IDP.
In particular, the following statement from \Cref{sec:cayley} is stronger than the afore-mentioned result of \cite{Fakhruddin,HNPS2008,Ogata}:

\begin{theorem}
\label{thm:cayley}
Let $Q$ be lattice polygon, and $P$ a weak Minkowski summand of $Q$. Then the Cayley sum $\cayley(P,Q)$ has a unimodular cover.
\end{theorem}

This has the following two corollaries, also proved in \Cref{sec:cayley}.
A  \emph{prismatoid} is a polytope whose vertices all lie in two parallel facets. 
A polytope has width $1$ if its vertices lie in two \emph{consecutive} parallel lattice hyperplanes. Observe that this is the same as being ($SL(\Z,d)$-equivalent to) a Cayley sum.

\begin{corollary}
\label{coro:prismatoid}
Every smooth $3$-dimensional lattice prismatoid has a unimodular cover.
\end{corollary}

\begin{corollary}
\label{coro:width1}
Every integer dilation $kP$, $k\ge 2$, of a lattice $3$-polytope $P$ of width $1$ has a unimodular cover.
\end{corollary}

A special case of the latter are integer dilations of empty tetrahedra. That their dilations have
unimodular covers is \cite[Cor.~4.2]{SantosZiegler} (and is also implicit in \cite{KantorSarkaria}).

\medskip
We believe that the $3$-polytopes in all these statements have unimodular triangulations, but this remains an open question.

\subsection*{Acknowledgements:} We would like to thank Akiyoshi Tsuchiya, Spencer Backman, and Johannes Hofscheier for posing these questions to us and 
Christian Haase for helpful discussions.

\section{Parallelepipeds}
\label{sec:parallelepipeds}

The main tool for the proof of \Cref{thm:parallelepipeds} is what we call the parallelepiped circumscribed to a given tetrahedron, defined as follows:

\begin{definition}
\label{def:circunpara}
Let $T$ be a tetrahedron with vertices $p_1$, $p_2$, $p_3$, and $p_4$. Consider the points $q_i= \frac12 (p_1+p_2+p_3+p_4) - p_i$, $i\in [4]$, and let
\[
C(T)=\conv(p_i,q_i: i\in[4]).
\] 
$C(T)$ is a parallelepiped with facets $\conv(p_i, p_j, q_k, q_l)$ for all choices of $\{i,j,k,l\}=[4]$. We call it the \emph{parallelepiped circumscribed} to $T$.

For each $i \in [4]$, let $T_i=\conv(q_i, p_j, p_k, p_l)$, with $\{i,j,k,l\}=[4]$; we call these $T_i$ the \emph{corner tetrahedra} of $C(T)$. Together with $T$ they triangulate $C(T)$.
\end{definition}

Modulo an affine transformation, the situation of $T$ and $C(T)$ is exactly that of the regular tetrahedron inscribed in a cube; see \Cref{fig:circumscribed_parall}. 
\begin{figure}[htb]
\includegraphics[scale=.25]{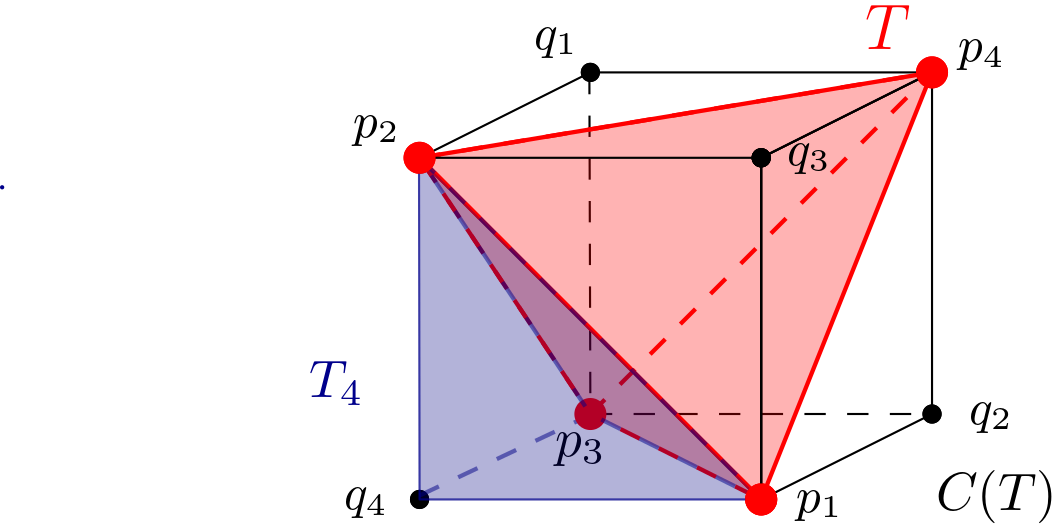}
\caption{In red we have a tetrahedron $T$, in black its circumscribed parallelepiped $C(T)$, and in blue the corner simplex $T_4$.}
\label{fig:circumscribed_parall}
\end{figure}

\begin{lemma}
\label{lemma:corner}
Let $T=\conv\{p_1,p_2,p_3,p_4\}$ be an empty lattice tetrahedron that is not unimodular. Let $C(T)$ be the parallelepiped circumscribed to $T$ and let $T_1, T_2,T_3$ and $T_4$ be the corresponding corner tetrahedra in $C(T)$. Then, every $T_i$ contains at least one lattice point different from $\{p_1,\dots,p_4\}$.
\end{lemma}

\begin{proof}
By White's classification of empty tetrahedra (\cite{White1964}, see also, e.~g.~\cite[Sect.~4.1]{HPPS-survey}), there is no loss of generality in assuming $T=\conv(p_1,p_2,p_3,p_4)$ with
\[
p_1=(0,0,0), \quad
p_2=(1,0,0), \quad
p_3=(0,0,1), \quad
p_4=(a,b,1).
\]
where $b\ge 2$ is the (normalized) volume of $T$, and $a\in \{1,\dots,b-1\}$ satisfies $\gcd(a,b)=1$. This gives 
\begin{align*}
q_1=\left(\frac{1+a}2,\frac{b}2,1\right), &&
q_2=\left(\frac{a-1}2,\frac{b}2,1\right), \quad\\
q_3=\left(\frac{1+a}2,\frac{b}2,0\right), &&
q_4=\left(\frac{1-a}2,-\frac{b}2,0\right).
\end{align*}

Then, the inequalities $b\ge 1+a \ge 2$ imply:
\[
u:=(1,1,0)\in \conv(p_1p_2q_3) \subset T_4, \quad
v:=(0,-1,0)\in \conv(p_1p_2q_4) \subset T_3.
\]

Observe that $u+v=p_1+p_2=q_3+q_4$.
Now, this implies that the quadrilateral $\conv(p_1q_4p_2q_3)$ contains a fundamental domain for the lattice $\Z^2\times\{0\}$. Hence, its translate $\conv(q_2p_3q_1p_4)$ contains a fundamental domain for $\Z^2\times\{1\}$ and, in particular, it contains at least one lattice point other than $p_3$ and $p_4$. By central symmetry around its center $\left(\frac{a}2,-\frac{b}2,1\right)$, $\conv(q_2p_3q_1p_4)$ must contain lattice points in both triangles $\conv(q_2p_3p_4)\subset T_1$ and $\conv(q_1p_3p_4)\subset T_2$.
%


\end{proof}

\begin{lemma}
\label{lemma:3<4}
Let $P$ be a lattice parallelepiped and let $T\subset P$ be a tetrahedron. Then, at least one of the four corner tetrahedra $T_i$ of the circumscribed parallelogram $C(T)$ is fully contained in $P$.
\end{lemma}

\begin{proof}
Let us denote the vertices of $T$ by $p_1, p_2, p_3, p_4$ and the vertices of $C(T)$ not in $T$ by $q_1, q_2, q_3, q_4$, with the conventions of \Cref{def:circunpara}. 

We call \emph{band} any region of the form $f^{-1}([\alpha,\beta])$ for some functional $f\in (\R^3)^*$ and closed interval $[\alpha,\beta]\subset \R$.
We claim that any band containing $T$ must contain at least three of the $q_i$s. 
This claim implies that the parallelepiped $P$, which is the intersection of three bands, contains at least one of the $q_i$s and hence it fully contains the corresponding $T_i$.

To prove the claim, suppose that $q_1\not\in B:= f^{-1}([\alpha,\beta])$ for a certain band $B \supset T$. 
Without loss of generality, say $f(q_1)<\alpha$. Then the equalities $q_1+q_2=p_3+p_4$ and $q_1+p_1=q_2+p_2$ respectively give:
\begin{gather}
\label{eq:first}
f(q_2) = f(p_3+p_4-q_1) = f(p_3)+f(p_4)-f(q_1) > 2\alpha-\alpha=\alpha,\\
\label{eq:second}
f(q_2) = f(q_2+p_2-p_1) = f(q_2)+(f(p_2)-f(p_1)) < \alpha + (\beta-\alpha) = \beta,
\end{gather}
so that $q_2 \in B$.

Inequality \eqref{eq:first} also implies
\begin{equation}
\label{eq:third}
f(q_1) <  f(p_i) < f(q_2), \quad\text{ for $i=3,4$}.
\end{equation}

The translation of vector $\frac12 (p_1+p_2-p_3-p_4)$ sends $q_1,q_2,p_3,p_4$ to $p_2,p_1,q_4,q_3$ (in this order). By applying this to inequality \eqref{eq:third}, we obtain
\[
\alpha \le f(p_2) < f(q_i) < f(p_1) \le \beta, \quad\text{ for $i=3,4$},
\]
so that $q_3,q_4 \in B$.
This finishes the proof of the claim, and of the lemma.
\end{proof}

\begin{corollary}
\label{coro:coverpara}
Let $T$ be an empty lattice tetrahedron contained in a lattice parallelepiped $P$. Then, $T$ can be covered by unimodular tetrahedra contained in $P$.
\end{corollary}

\begin{proof}
We proceed by induction on the (normalized) volume of $T$, which is a positive integer. If this volume equals $1$ then $T$ is unimodular and there is nothing to prove, so we assume $T$ is not unimodular. Let $p_1, p_2, p_3, p_4$ denote the vertices of $T$.

\Cref{lemma:3<4} guarantees that one of the corner tetrahedra $T_i$ of the parallelepiped $C(T)$ is contained in $P$. Without loss of generality, suppose $T_4 = \conv(p_1, p_2, p_3,q_4)$ is in $P$. By \Cref{lemma:corner}, we know that $T_4$ contains a lattice point other than the $p_i$s, which we denote by $u$. 
%
Then $S=\conv(T\cup \{u\})$ can be triangulated in two different ways: $S=T \cup T'_4$, where $T'_4 = \conv(p_1, p_2, p_3, u) \subseteq T_4$ and $S= S_1 \cup S_2 \cup S_3$, with
\[
S_1= \conv(p_2,p_3,p_4, u),
S_2=\conv(p_1,p_3,p_4, u),
S_3=\conv(p_1,p_2,p_4, u).
\]

Each of the tetrahedra $S_i$ has lattice volume strictly smaller than $T$ because, for each $i$, $p_i$ is the unique point of $C(T)$ maximizing the distance to the opposite facet $\conv(p_j,p_k,p_l)$ of $T$. Thus, $S_1$, $S_2$ and $S_3$ cover $T$ and have volume strictly smaller than $T$. The $S_i$ may not be empty, but we can triangulate them into empty tetrahedra, which by inductive hypothesis they can be covered unimodularly.
%
%
\end{proof}

\begin{proof}[Proof of \Cref{thm:parallelepipeds}]
Arbitrarily triangulate the parallelepiped into empty lattice tetrahedra and apply \Cref{coro:coverpara} to these tetrahedra.
%
%
\end{proof}


Let us say that a lattice $3$-polytope $P$ \emph{has the circumscribed parallelepiped property} if it satisfies the conclusion of \Cref{lemma:3<4}:  ``for every empty tetrahedron $T$ contained in $P$ at least one of the four corner tetrahedra in $C(T)$ is 
contained in $P$''. 
If this holds then $P$ has a unimodular cover, since then the proofs of \Cref{coro:coverpara} and \Cref{thm:parallelepipeds} work for $P$.
Hence, a positive answer to the following question would imply that every smooth $3$-polytope has a unimodular cover, which in turn implies \Cref{conj:Oda}\eqref{itm:smoothIDP} in dimension three.

\begin{question}
Does every smooth 3-polytope have the circumscribed parallelepiped property? 
\end{question}

Our proof that parallelepipeds have the property (\Cref{lemma:3<4}) is based on the fact that they have only three (pairs of) normal vectors. The proof, and the property of being IDP, fail if there are four of them:

\begin{example}[Non-IDP octahedron and triangular prism]
\label{ex:non-IDP}
The following lattice octahedron $Q$ and triangular prism $P$ are not IDP:
\begin{gather}
Q= \conv((0,1,1),(1,0,1),(1,1,0),(0,-1,-1),(-1,0,-1),(-1,-1,0)),\\
P=\conv((0,1,1),(1,0,1),(1,1,0),(-1,0,0),(0,-1,0),(0,0,-1)).
\end{gather}

Indeed, in both polytopes the only lattice points are the six vertices and the origin. The point $(1,1,1)$ lies in the second dilation but is not the sum of two lattice points in the polytope. Hence, they are not IDP, which implies they do not admit unimodular covers.
\end{example}

\section{ Cayley sums}
\label{sec:cayley}
Let $P$ and $Q$ be two lattice polytopes in $\R^d$. We do not require them to be full-dimensional, but we assume their Minkowski sum is. Remember that the \emph{Minkowski sum} $P+Q$ and the \emph{Cayley sum} of $P$ and $Q$ are defined as:
\begin{gather*}
P + Q := \{ p+q \in \R^d: p\in P, q \in Q\} \subset \R^d,\\
\cayley (P,Q) = \conv( P\times\{0\} \cup Q\times \{1\}) \subset \R^{d+1}.
\end{gather*}

The so-called \emph{Cayley Trick} is the isomorphism
\[
2\cayley(P, Q) \cap (\R^d\times \{1\}) \cong P+Q,
\]
which easily implies:

\begin{proposition}[see, e.g.~\protect{\cite[Thm.~0.4]{Tsuchiya}}]
\label{prop:mixedIDP}
If $\cayley(P,Q)$ is IDP then the pair $(P,Q)$ is mixed IDP.
\end{proposition}


The Cayley Trick also provides the following canonical bijections:

\[
\begin{array}{ccc}
\text{polyhedral subdivisions of $\cayley(P,Q)$} &\leftrightarrow& \text{mixed subdivisions of $P + Q$}\\
\text{triangulations of $\cayley(P,Q)$} &\leftrightarrow& \text{fine mixed subdivisions of $P + Q$}\\
\text{unimodular simplices in $\cayley(P,Q)$} &\leftrightarrow& \text{unimodular prod-simplices in $P + Q$}.
\end{array}
\]
See \cite{DLRS2010} for more details on the Cayley Trick and on triangulations and polyhedral subdivisions of polytopes.
In fact these bijections can be taken as definitions of the objects in the right-hand sides. In particular, 
we call \emph{prod-simplices} in $P+Q$ the Minkowski sums $T_1+T_2$ where $T_1\subset P$ and $T_2\subset Q$ are simplices with complementary affine spans. A prod-simplex is \emph{unimodular} if the edge vectors from a vertex of $T_1$ and from a vertex of $T_2$ form a unimodular basis.

\medskip

We now turn our attention to $d=2$, in order to prove \Cref{thm:cayley}. A triangulation of $\cayley(P,Q)\subset \R^3$ consists of tetrahedra of types $(1,3)$, $(2,2)$ and $(3,1)$, where the type denotes how many vertices they have in $P$ and in $Q$. Empty tetrahedra of types $(1,3)$ or $(3,1)$, which are Cayley sums of a triangle in $P$ and a point in $Q$, or viceversa, are automatically unimodular. The case that we need to study are therefore tetrahedra of type $(2,2)$, which are Cayley sums of a segment $p\subset P$ and a segment $q\subset Q$.
The following lemma, whose proof we postpone to \Cref{sec:the_lemma}, is crucial to understand how to unimodularly cover these tetrahedra. 
We use the following conventions: if $a, b$ are points, we denote by $[a,b]$ and $(a,b)$ respectively the closed and open segments with endpoints $a,b$. Given a segment $s=[a,b]$, we denote the vector $\vec s:= b-a$ and  the line spanned by $\vec s$ by $\vecline s$.

\begin{lemma}
\label{lemma:cayley}
Let $Q$ be a two-dimensional lattice polytope and $P$ a weak Minkowski summand of it. 
Let $p=[p_1,p_2] \subset P$ and $q=[q_1,q_2]\subset Q$ be two primitive and non-parallel lattice segments, and let $\vecline p$ and $ \vecline q$ be the lines spanned by them.  If the parallelogram $p + q$ is not unimodular, then at least one of the regions
\[
((p_1, p_2) + \vecline q ) \cap P, 
\qquad \text{and} \qquad
((q_1, q_2) + \vecline p ) \cap Q
\]
contains a lattice point. See \Cref{fig:strips}.
\end{lemma}

\begin{figure}[htb]
\scalebox{.75}{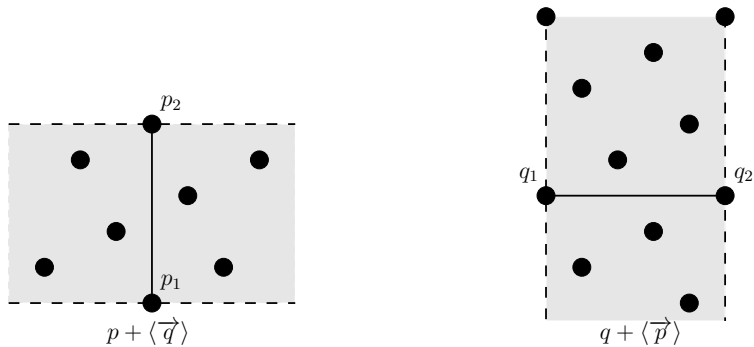}
\caption{The strips  of Lemma \ref{lemma:cayley}}
\label{fig:strips}
\end{figure}

\begin{corollary}
\label{coro:covercayley}
Let $T$ be an empty lattice tetrahedron contained in the Cayley sum $\cayley(P,Q)$, where $Q$ is a lattice polygon and $P$ is a weak Minkowski summand of $Q$. Then, $T$ can be covered by unimodular tetrahedra contained in $\cayley(P,Q)$.
\end{corollary}

\begin{proof}
The proof is by induction on the normalized volume of $T$, which we assume to be at least $2$. This implies that $T$ is of type $(2,2)$,
since empty tetrahedra of types $(1,3)$ and $(3,1)$ are unimodular. Thus, $T$ is the Cayley sum of primitive segments $p=[p_1,p_2]\subset P$ and $q=[q_1,q_2]\subset Q$.  
Let $u$ be the lattice point whose existence is guaranteed by \Cref{lemma:cayley}. Assume  (the other case is similar) that 
\[
u \in ((p_1, p_2) + \vecline q ) \cap P,
\]
and call $t$ the triangle $t=\conv( u,  p_1,  p_2)\subset P$.

Let us denote $\tilde u$, $\tilde p_1$, $\tilde p_2$, $\tilde q_1$, $\tilde q_2$ the points corresponding to $u, p_1, p_2, q_1, q_2$ in $\cayley(P,Q)$.
That is, $\tilde p_i = p\times\{0\}$, $\tilde q_i = p\times\{1\}$, and $\tilde u = u\times\{0\}$.
Observe that the assumption $u\in((p_1, p_2) + \vecline q$ implies that of the segments $[u,q_i]$ crosses the triangle $\conv(p_1,p_2,q_j)$, where $\{i,j\}=\{1,2\}$, see \Cref{fig:flip}. 

\begin{figure}[htb]
\includegraphics[scale=.3]{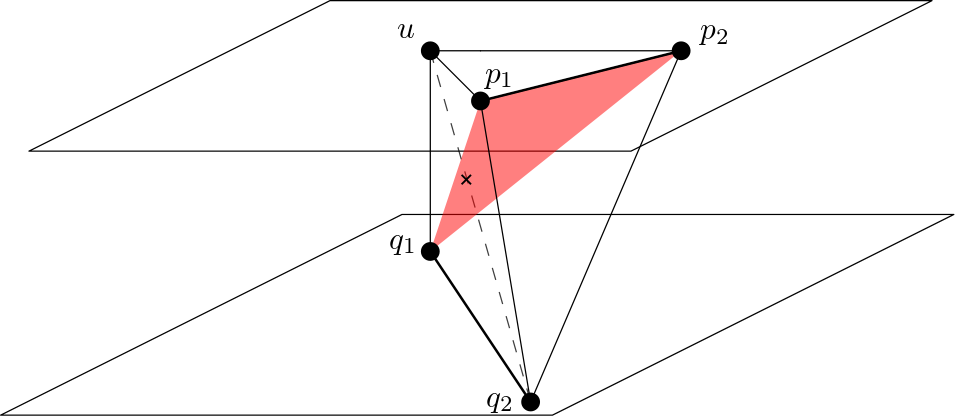}
\caption{$[u,q_2]$ intersects $\conv(p_1,p_2,q_1)$}
\label{fig:flip}
\end{figure}

In turn, this means that the polytope $\conv(\tilde u, \tilde p_1, \tilde p_2, \tilde q_1, \tilde q_2) = \cayley(t,q)$
 has the following two triangulations:
\begin{gather*}
\mathcal T^+:= \left\{ \cayley(p,q), \cayley(t, \{q_i\}) \right\},
\\
\mathcal T^-:= \{ \cayley([p_1,u],q), \cayley([p_2,u],q), \cayley(t, \{q_j\}) \}.
\end{gather*}
The tetrahedra $\cayley(t, \{q_j\})$ and $\cayley(t, \{q_i\})$ are unimodular, which implies that $T=\cayley(p,q)$ has volume equal to the sum of the volumes of $\cayley([p_1,u],q)$ and $\cayley([p_2,u],q)$. In particular, we have covered $T$ by the three tetrahedra in $\mathcal T^-$, which are of smaller volume and hence have unimodular covers by inductive assumption.
%
%
\end{proof}

\begin{proof}[Proof of \Cref{thm:cayley}]
Arbitrarily triangulate $\cayley(P,Q)$ into empty lattice tetrahedra and apply \Cref{coro:covercayley} to these tetrahedra.
\end{proof}

Let us now show how to derive \Cref{coro:prismatoid,coro:width1} from this theorem.
\emph{Prismatoids} were defined in~\cite{Santos-hirsch} as polytopes whose vertices all lie in two parallel facets. In particular, a \emph{lattice prismatoid} is any $d$-polytope $SL(\Z,d)$-equivalent to one of the form
\[
\conv(Q_1\times\{0\} \cup Q_2 \times \{k\}),
\]
where $Q_1,Q_2$ are lattice $(d-1)$-polytopes and $k\in \Z_{>0}$. This is almost a generalization of Cayley sums, which would be the case $k=1$, except the definition of prismatoid requires $Q_1$ and $Q_2$ to be full-dimensional, while the Cayley sum only requires this for $Q_1+Q_2$.

\begin{proposition}
\label{prop:prismatoid}
Let $Q_1$, $Q_2$ be two lattice polygons and consider the prismatoid 
\[
P:= \conv(Q_1\times\{0\} \cup Q_2 \times \{k\},
\]
with $k\ge 2$. 
If $P\cap(\R^2\times\{1\})$ is a lattice polygon then $P$ has a unimodular cover.
\end{proposition}

\begin{proof}
The condition that $P\cap(\R^2\times\{1\})$ is a lattice polygon implies the same for $P\cap(\R^2\times\{i\})$, for every $i$. 
Indeed, the condition implies that every edge of $\cayley(P,Q)$ of the form $[u\times \{0\}, v\times \{k\}]$ has a lattice point in $\R^2\times\{i\}$, and hence it has a lattice point in $P\cap(\R^2\times\{i\})$, for every $i$.

Observe that for every $i\in \{1,\dots,k-1\}$ the intersection $P\cap(\R^2\times\{i\})$ has the same normal fan as $Q_1+Q_2$. Thus, each slice
\[
P \cap (\R^2\times[i-1,i])
\]
is a Cayley polytope. For $i\in\{2,\dots,k-1\}$,  both bases have the same normal fan (and therefore each is a weak Minkowski summand of the other); for $i\in \{1,k\}$ one base is a weak Minkowski summand of the other. We can therefore apply \Cref{thm:cayley} to each slice and combine the covers thus obtained to get a unimodular cover of $P$.
\end{proof}

\begin{proof}[Proof of \Cref{coro:prismatoid,coro:width1}]
In both cases the polytope under study satisfies the hypotheses of \Cref{prop:prismatoid}: in \Cref{coro:prismatoid}, the smoothness of the prismatoid implies that every edge of the form $[u\times \{0\}, v\times \{k\}]$ has lattice points in all slices. In \Cref{coro:width1}, since $P$ has width one, $P\cong \cayley(Q_1, Q_2)$ for some $Q_1$ and $Q_2$. Hence,
\[
kP \cap(\R^2\times\{1\}) = (k-1)Q_1 + Q_2.
\qedhere
\] 
\end{proof}

\section{Proof of \Cref{lemma:cayley}}
\label{sec:the_lemma}
Let $f_q$ be the primitive lattice functional constant on $q$ and $f_p$ the one constant on $p$. We assume that $f_q(p_1) < f_q(p_2)$ and $f_p(q_1) < f_p(q_2)$.

Observe that in the strip $q +\vecline p$, there is a unique lattice point on the line $f_q(x)=-1$;  indeed, since $q$ is primitive, the only way that in the strip there could be two lattice points on $f_q(x)=-1$ is if they were on the boundary of the strip, which would however imply that $p+q$ is a unimodular paralellogram, against our assumptions.
Since translating the polytopes by lattice vectors will not result in any loss of generality, we can assume that $p_1$ is that unique lattice point. That is, $f_q(p_1)=-1$, or equivalently, the triangle $\conv(q_1, q_2, p_1)$ is unimodular. Similarly, the unique lattice point in the strip on the line $f_q(x)=1$ is then $q_1+q_2 -p_1$.

We let $H_1=\{f_q(x) \leq 0\}$ and $H_2=\{f_q(x) \geq 0\}$; similarly let $V_1=\{f_p(x) \leq 0\}$ and $V_2=\{f_p(x) \geq 0\}$. 
In the figures, we draw $p$ as a vertical segment and $q$ as a horizontal one, so that $H_i \cap V_j$ are the four quadrants. 
See \Cref{fig:setup}.
\begin{figure}[htb]
\includegraphics[scale=.3]{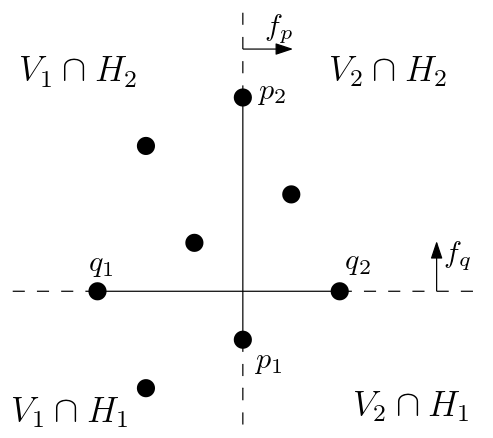}
\caption{Setup for the proof of \Cref{lemma:cayley}}
\label{fig:setup}
\end{figure}

Let $w=\area(p+q) \geq 2$, where $\area$ denotes the area normalized to a fundamental domain. Then:
\[
w=\width_{f_q}(p + \vecline q )=\width_{f_q}(p)=\width_{f_p}(q)=\width_{f_p}(q +\vecline p ). 
\]

\begin{proof}[Proof of \Cref{lemma:cayley}]
Suppose by contradiction that there is no lattice point as described in the lemma. In particular, no lattice point on the boundary of $Q$ can be in the interior of the strip $q + \vecline p$.  Thus the boundary of $Q$ contains two primitive segments which each have one vertex on each side of the strip $q + \vecline p$; we will call these  $b=[b_1, b_2], t=[t_1, t_2]$, with $b$ and $t$ crossing the strip in $H_1$ and $H_2$ respectively and the convention that $f_p(b_2) >f_p(b_1)$ and $f_p(t_2) >f_p(t_1)$. This readily implies 
\begin{gather}
\label{eq:widthq}
\begin{array}{cc}
f_p(t_1) \leq f_p(q_1), &
f_p(t_2) \geq f_p(q_2), \\
f_p(b_1) \leq f_p(q_1), &
f_p(b_2) \geq f_p(q_2).
\end{array}
\end{gather}

The same holds for $P$ and the strip $p+\vecline q$, and we call the segments $\ell=[l_1, l_2]$ and $r=[r_1, r_2]$, with $\ell$ and $r$ crossing the strip $p + \vecline q$  in $V_1$ and $ V_2$ respectively. The only difference is that in the case that $P$ is one dimensional we have $\ell=r=p$.  Again we have
\begin{gather}
\label{eq:widthp}
\begin{array}{cc}
f_q(l_1) \leq f_q(p_1), &
f_q(l_2) \geq f_q(p_2), \\
f_q(r_1) \leq f_q(p_1),&
f_q(r_2) \geq f_q(p_2).
\end{array}
\end{gather}

Observe that a priori one of $l$ and $r$ can coincide with $p$, if this is on the boundary of $P$, and similarly one of $t,b$ might be $q$, if this is on the boundary of $Q$.

\begin{claim}
The following inequalities hold, 
\begin{align*}
\width_{f_q}(\ell) , 
\width_{f_q}(r) ,
\width_{f_p}(t) ,
\width_{f_p}(b) \geq w.
\end{align*}
Each inequality is strict, unless the segment in question coincides with $p$ or $q$.
\end{claim}

\begin{proof}
The inequality $\geq w$ follows in each case from \eqref{eq:widthp} and \eqref{eq:widthq}.

If one of the inequalities, say the one for $\ell$, is not strict, then $\ell$ has one endpoint on each of the boundary lines of $(p + \vecline q)$. Unless $\ell = p$, one of the endpoints of $\ell$ is not an endpoint of $p$, say $l_1 \neq p_1$. Thus the triangle $T=\conv(p_2, p_1, l_1)$ is contained in $P$ and its edge $[p_1, l_1]$ is an integer dilation of $q$. Since $\width_{f_q}(T) =w \geq 2$, $T$ must contain a lattice point in the interior of the strip.
\end{proof}

%

\begin{claim}
\label{claim:b_and_t}
$f_q(b_2-b_1)$ and $f_q(t_2 - t_1)$ are non-zero and have the same sign. That is, $f_q$ achieves its maximum over $b$ and over $t$ on the same halfplane $V_1$ or $V_2$.
\end{claim}

%

\begin{proof}
Suppose by contradiction that the maximum of $f_q$ on $t$  lies in $V_1$ and that the maximum on $b$ lies in $V_2$. 

Then $Q \cap V_2$ is contained in the open strip $\{-1<f_q(x)<w-1\}$, of width $w$. This cannot contain a translated copy of $r$, since $\width_{f_q}(r) \geq w$, see \Cref{fig:claim2}. This is a contradiction, since $P$ is a Minkowski summand of $Q$ and therefore $Q$ must have an edge parallel to $t$.
\end{proof}

\begin{figure}[htb]
\scalebox{.75}{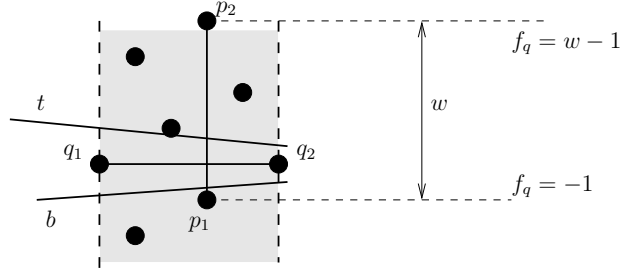}
\caption{Illustration of the proof of \Cref{claim:b_and_t}}
\label{fig:claim2}
\end{figure}

We assume w.l.o.g. that the maximum on $t$ (and hence on $b$) is achieved in $V_2$, that is to say, $f_p$ and $f_q$ increase in the same direction along $t$ (and hence along $b$). 

\begin{claim}
\label{claim:r}
Assume w.l.o.g.~that $b$ and $t$ either are parallel or their affine spans cross in $V_2$. Then, 
\begin{enumerate}
\item The intersection of $Q$ with any line parallel to $p$ in $V_2$ has width w.r.t. $f_q$ strictly smaller than $w$.
\item $f_p(r_2) > f_p(r_1)$, that is, $f_p$ achieves its maximum over $r$ in $H_2$.
\end{enumerate}
\end{claim}

\begin{proof}
Both $t$ and $b$ must intersect $p$, otherwise $p_1$ or $p_2$ are the lattice points we are looking for in $Q$.  Their intersections with $p$ are thus endpoints of a segment of width w.r.t $f_q$ less than $w$, the width of $p$. Since $t$ and $b$ cross in $V_2$, the same is true for any segment parallel to $p$ contained in $Q \cap V_2$.  

\begin{figure}[htb]
\scalebox{.75}{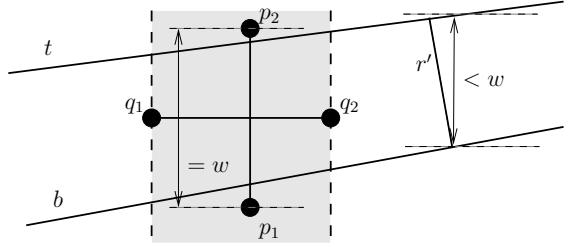}
\caption{Illustration of the proof of \Cref{claim:r}}
\label{fig:claim3}
\end{figure}

For part (b), If $f_p(r_2) \leq f_p(r_1)$, it would be impossible to fit a translated copy $r'$ of $r$ in the correct side of $Q$: $r$ would need to lie inside the triangle delimited by the affine line $\langle t \rangle$ and the inequalities $f_q(x) \geq f_q(r_1)$, $f_p(x) \leq f_p(r_1)$. However, this triangular region has width less than $w$  w.r.t. $f_q$, by combining part (a) with the fact that $f_p$ and $f_q$ increase in the same direction along $t$, see \Cref{fig:claim3}. 
\end{proof}

The last two claims can be summarized as saying that in the pictures $b$, $t$ and $r$ have positive slope. Observe that this implies that $q$ is not in the boundary of $Q$ and $p \neq r$, so both $P$ and $Q$ are full dimensional.

%
%
%

Let $g$ be the primitive lattice functional constant on $[p_1, q_2]$ (and therefore constant also on $[q_1, q_1+q_2-p_1]$). By the assumption on $p_1$, the values of $g$ on these segments differer by $1$. We choose the sign of $g$ so that 
\[
g([p_1, q_2])= g( [q_1, q_1+q_2-p_1]) +1. 
\]

\begin{claim}
\label{claim:g}
$g(t_1) > g(t_2)$, $g(b_1) > g(b_2)$, and $g(r_1) < g(r_2)$.
\end{claim}

\begin{proof}
Since $b$ and $t$ must respectively separate $p_1$ and $q_1+q_2-p_1$ from the other two vertices of the parallelogram $\conv(q_1, p_1, q_2, q_1+q_2-p_1)$, they must respectively intersect its (parallel) edges $[p_1, q_2]$ and $[q_1, q_1+q_2-p_1]$, which implies the stated inequalities for $b$ and $t$.
The same argument  applied to the parallelogram  $\conv(p_1, q_2, p_2, p_1+p_2-p_2)$, yields the inequalities for $\ell$ and $r$.
\end{proof}

\begin{figure}[htb]
\scalebox{.75}{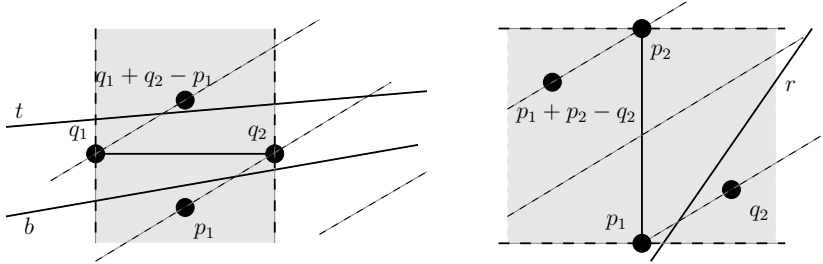}
\caption{Illustration of the proof of \Cref{claim:g}}
\label{fig:claim4}
\end{figure}

We are now ready to show a contradiction. Since the normal fan of $Q$ refines that of $P$, $Q$ must have an edge $r'$ which is a translated copy of $r$. Let $r_1'$ and $r_2'$ be its endpoints. Now consider the lattice line $d$ through $r_1'$ parallel to  $[p_1, q_2]$, that is, $g$ is constant on $d$. Let $d'$ be the parallel line defined by $g(d')=g(d)+1$.

Consider the segment $s$ contained in $r_1' + \vecline p$ with endpoints $s_1=r_1'$ on $d$ and $s_2$ on $d'$.  
Since $t$ separates $q_1$ and $q_1+q_2-p_1$ and $g$ decreases from $t_1$ to $t_2$ (by \Cref{claim:g}), the inequality $g(x)< g(d')$ holds on $Q\cap V_2$,
and in particular for $r_2'$. Since $r_2'$ is a lattice point, $g(r_2')\leq g(d)= g(r_1')$, which contradicts \Cref{claim:g}).
\end{proof}

\end{document}